\documentclass[a4paper,11pt]{amsart}
\usepackage{fullpage}
\usepackage{amsmath}
\usepackage{amsthm}
\usepackage{amssymb}
\usepackage{physics}
\usepackage{graphicx}
\usepackage[toc,page]{appendix}
\usepackage{floatrow}
\usepackage{mathtools}
\usepackage{tikz-cd} 
\usepackage{todonotes}
\usepackage[all]{xy}
\usepackage{bbm}
\usepackage{soul}
\usepackage[T1]{fontenc}
\usepackage{romannum}
\usepackage{todonotes}
\usepackage{framed}

\usepackage{hyperref}
\hypersetup{
    colorlinks=true,
    linkcolor=blue,
    filecolor=magenta,      
    urlcolor=cyan,
    pdftitle={Overleaf Example},
    pdfpagemode=FullScreen,
    }

\newtheorem{theorem}{Theorem}[section]

\newtheorem{definition}[theorem]{Definition}

\newtheorem{cor}[theorem]{Corollary}
\newtheorem{prop}[theorem]{Proposition}
\newtheorem{lemma}[theorem]{Lemma}
\newtheorem{example}[theorem]{Example}

\newtheorem{remark}[theorem]{Remark}
\newtheorem{conj}[theorem]{Conjecture}

\newtheorem*{rem}{Remark}

\newtheorem*{theoremA}{Theorem A}
\newtheorem*{theoremC}{Theorem C}
\newtheorem*{corB}{Corollary B}

\newcommand{\C}{\mathbb{C}}

\newcommand{\Cmc}{\C [z_1,\ldots,z_d]} 

\newcommand{\HS}{\mathcal{H}} 

\newcommand{\Rd}{\mathcal{R}_{d}}

\title{The Geometric Arveson--Douglas Conjecture through Veronese Embeddings}

\author{Francesca Arici}
\address{Mathematical Institute, Leiden University, P.O. Box 9512, 2300 RA
Leiden, the Netherlands}
\email{f.arici@math.leidenuniv.nl}
\thanks{Partially funded by the Netherlands Organisation of Scientific Research (NWO) under the grant VI.Vidi.233.231 }

\author{Yufan Ge}
\address{Mathematical Institute, Leiden University, P.O. Box 9512, 2300 RA
Leiden, the Netherlands}
\subjclass[2020]{Primary 47A13; Secondary 14N05, 19K35, 47B35, 46L80}
\keywords{Arveson--Douglas conjecture, Drury--Arveson space, equivariant $KK$-theory, projective variety, subproduct system, Toeplitz extension, Veronese embedding}
\curraddr{}
\email{y.ge@math.leidenuniv.nl, yufan.ge.math@qq.com}
\thanks{}

\author{Dimitris M. Gerontogiannis}
\address{IM PAN,  Jana i Jędrzeja Śniadeckich 8, 00-656 Warszawa, Poland}
\curraddr{Mathematical Institute, Leiden University, P.O. Box 9512, 2300 RA
Leiden, the Netherlands}
\email{dgerontogiannis@impan.pl}
\thanks{}

\begin{document}
\pagenumbering{arabic}
    
\begin{abstract}
We prove invariance of the geometric Arveson--Douglas conjecture under Veronese embeddings of projective varieties. This provides new examples of singular varieties satisfying the conjecture. Furthermore, it reduces the conjecture to projective varieties whose defining ideals admit quadratic Gr{\"o}bner bases, by keeping the same Schatten regularity thresholds. Our results lead to canonical examples of subproduct-systems whose Toeplitz algebras are not $KK$-equivalent to the complex numbers. Finally, we show that Douglas' fundamental class conjecture, asserting that the associated Toeplitz extension implements $\mathbb T$-equivariant duality in $KK$-theory, is invariant under Veronese embeddings as well. While the latter result bootstraps the fundamental class conjecture, it also leads to infinitely many counterexamples. 
\end{abstract}

\maketitle

\section{Introduction}
The Arveson--Douglas conjecture \cite{WA05, RD06} is one of the central open problems at the
interface of multivariable operator theory, algebraic geometry, and
noncommutative geometry. It goes back to Arveson's work \cite{WA97,WA05} on essential normality of commuting row-contractions on Hilbert spaces of analytic functions, and to Douglas' observation that these concepts potentially extend index theory to singular projective varieties \cite{RD06}. There are several related formulations of the conjecture, e.g. on Hardy space, (weighted) Bergman space, and Drury--Arveson space \cite{Sh14,GW20}. Here, we focus on the latter, which is particularly challenging; see, for instance, \cite{WX21}.
 
Let $H_d^2$ be the Drury--Arveson space in $d$ commuting variables, which is a left module over the ring $\Rd:=\mathbb C[z_1,\ldots, z_d]$ of complex polynomials in $d$-variables, with action by bounded multiplication operators. Also, for a homogeneous ideal $I\triangleleft  \Rd$ and $1\leq i\leq d$, let $R_i$ denote the compression of the (left) multiplication operator by $z_i$ on $H_d^2$ to the quotient $Q_{I}:=H_d^2/ \overline{I}$, where $\overline{I}$ is the closure of $I$. Further,  let $\mathcal{S}_p(H)$ be the $p$-Schatten ideal ($p\geq 1$) on a separable Hilbert space $H$, with $\mathcal S_{\infty}(H)$ denoting the compact operators $\mathbb K(H)$.
 
Following Arveson \cite{WA05}, we say that the quotient module $\mathcal{Q}_I$ of the Drury--Arveson space $H^2_d$ is $p$-essentially normal, if the commutators $ [R_i,R_j^*]$ are in the Schatten class $S_p(Q_I)$
for all $i,j =1, \dots, d$.

The Arveson--Douglas conjecture concerns the essential normality of the quotient module of a homogeneous ideal (see Conjecture \ref{conj:ADmain}). In geometric terms, the Arveson--Douglas conjecture can be reformulated using projective varieties and their affine cone.
 
\begin{conj}[Geometric Arveson--Douglas (GAD) conjecture]\label{conj:GADintro}
Let $X\subset \mathbb P^{d-1}$ be a projective variety. Denote by $\widehat{X}\subset \mathbb C^d$ its affine cone. Then, the quotient module $Q_{X}:=Q_{I(X)}$ is $p$-essentially normal for every $p>\dim_{\mathbb C}\widehat{X}$. 
\end{conj}
 
The conjecture has been confirmed for monomial ideals by Arveson \cite{WA05}; for principal ideals, affine cones of complex dimension at most one, and ambient dimension $d\leq 3$ by Guo--Wang \cite{GW08}; for finite unions of linear subspaces by Kennedy--Shalit \cite{KS12}; and for affine cones smooth away from the origin by Engli\v{s}--Eschmeier \cite{EE15} (see also Douglas--Tang--Yu \cite{DTY16}). Parallel Bergman-space results form a large part of the literature as well; see \cite{Har22, GW20} and the references therein.
 
In this work, we show that the GAD conjecture is invariant under Veronese embeddings. Namely, for $k\geq 1$, let $N_k={d+k-1\choose k}$ and $\nu_k:\mathbb P^{d-1}\to \mathbb P^{N_k-1}$ be the normalised Veronese embedding, see Equation \eqref{eq:normVmap}. Our first main result is the following.
 
\begin{theoremA}[Theorem \ref{thm:VpowerAD}]
Let $X\subset \mathbb P^{d-1}$ be a projective variety and $k\geq 1$ be arbitrary. Then,
$$X\text{ satisfies GAD conjecture}\Leftrightarrow \nu_k(X)\text{ satisfies GAD conjecture}.$$
\end{theoremA}
 
It is well known that the Veronese power of a ring preserves the underlying projective scheme. Theorem A is indeed inspired by that fact, but the challenge here is to lift this at the level of Drury--Arveson modules. Specifically, the key tool in proving Theorem A identifies $Q_{\nu_k(X)}$ with the divisible-by-$k$ homogeneous summands of $Q_X$, which is an $\mathcal R_{N_k}$-submodule of $Q_X$, via an $\mathcal R_{N_k}$-module unitary. The gist is to build this unitary. To this end, we use the fact that $Q_I$ is a reproducing kernel Hilbert space when $I$ is radical. 
 
Theorem A also yields examples of modules that do not fit into the list of known cases of quotient modules for which the conjecture holds. More specifically, the Veronese embeddings define a bootstrapping operation which produces examples (e.g. see Example \ref{ex:1}) beyond the familiar low-dimensional, hypersurface, smooth, or linear situations. Moreover, older results of Arveson \cite{WA07} and Shalit \cite{OS11} reduced the general Arveson--Douglas conjecture to quadratic homogeneous polynomials, 
however, at the cost of increasing the Schatten regularity threshold and drastically changing the underlying geometry. In that light, Theorem A offers (although for radical ideals) a direct geometric reduction of the geometric conjecture to radical homogeneous ideals admitting a quadratic Gr{\"o}bner basis (see the discussion below for subproduct systems). Note that not every quadratic homogeneous ideal admits such a basis.
 
Theorem A, paired with Kennedy--Shalit's work \cite{KS15}, also offers a Veronese bootstrapping operation in the hyperrigidity of the canonical generating sets of the Toeplitz algebras associated with quotient modules of Drury--Arveson spaces. For more details we refer to Open Problem 2. 
 
We then move on to investigate the Veronese bootstrapping operation at the level of index theory, in particular in relation with the Toeplitz extensions associated with quotient modules. Specifically, if the GAD conjecture holds
for $X$ (even for $p=\infty$), then the $C^*$-algebra $\mathcal T(Q_{X})$ generated by the
compressed coordinate multipliers, gives the Toeplitz extension $\tau_X$,
\begin{equation}\label{eq:intro_toe}
0\longrightarrow \mathbb K(Q_{X})\longrightarrow \mathcal T(Q_{X}) \longrightarrow C(\widehat{X}\cap \partial \mathbb B_d)\longrightarrow 0.    
\end{equation}
We shall refer to the boundary space $L(Q_X):=\widehat{X}\cap \partial \mathbb B_d$ as the \textit{shadow space of $X$}.
 
The extension $\tau_X$ fits into the broader framework of subproduct systems, which provide a common framework for several Fock-type constructions. A subproduct system is a graded family $H=(H_n)_{n\geq 0}$, with $H_0=\mathbb C$, such that $H_{n+m}\subseteq H_n\otimes H_m$ associatively. They were introduced by Shalit--Solel \cite{SS09} and Bhat--Mukherjee \cite{BM10} in connection with CP-semigroups. In that generality, Viselter \cite{Vis12} showed that one always has a Toeplitz extension. In particular, for the Drury--Arveson quotients which are naturally subproduct systems, the GAD conjecture says that the boundary algebra is the commutative $C^*$-algebra $C(L(Q_X))$; the finer Schatten estimates provide connections to noncommutative geometry \'a la Connes.
 
It is worth noting that inspired by Manin's programme for noncommutative geometry \cite{Man88,Man91}, the first two authors of the current paper initiated the study of quadratic subproduct systems in \cite{AG26}. Around the same time, the notion of Veronese powers of subproduct systems was introduced by the second-named author in \cite{YG25}, as a Hilbert space analogue of Veronese subalgebras of graded algebras. A key feature of the Veronese power operation is that sufficiently high Veronese powers of a graded algebra admit quadratic Gr\"obner bases \cite{ERT94}, making it an effective tool for turning an arbitrary ideal into a quadratic one. 
Specifically, by \cite[Proposition~2.2, Remark 2]{PoPo}, for a commutative one-generated algebra $A$, the Veronese subalgebra $A^{(k)}$ is quadratic provided the algebra $A$ has no defining relations of degree $> 2k$. Another interesting direction is the study of $q$-deformed analogues of  the Drury--Arveson space arising from subproduct systems, see \cite{BGN25}. 

A first result related to the extension $\tau_X$ concerns Viselter's proposed series of open problems \cite{Vis12} connected to the theory of subproduct systems. One of them was the existence of a subproduct system whose Toeplitz algebra is not $KK$-equivalent to $\mathbb C$. Although subproduct systems were not yet introduced in the early 90's, this question also goes back to Lesch \cite{Les90}. Here, we provide the first such canonical example. Specifically, for $X=\nu_k(\mathbb P^1)$, we observe that the associated shadow space is the lens space $L(k,1)$. Also, the GAD conjecture holds due to smoothness, and therefore the Toeplitz extension exists. Then, from Pimsner's six-term exact sequence in $K$-theory, we obtain the following.

\begin{corB}[Corollary \ref{cor:B}]\label{cor:cor B}
We have that $\mathcal{T}(Q_{\nu_k(\mathbb P^1)})$ is not $KK$-equivalent to $\mathbb C$, if $k\geq 2$.   
\end{corB}
 
Following Douglas \cite{RD06}, $\tau_X$ defines an odd $K$-homology class $[\tau_X]\in KK_1(C(L(Q_X)),\mathbb C)$. In fact, as $\tau_X$ is $\mathbb T$-equivariant (this is straightforward, e.g. see \cite{DTY16}), if $m$ is the multiplication on $C(L(Q_X))$ and $\Delta_X:=m^*[\tau_X]$, one obtains the group homomorphism via Kasparov (slant) product,
\begin{equation}\label{eq:intro_duality}
-\otimes_{C(L(Q_X))} \Delta_X:KK^{\mathbb T}_*(\mathbb C, C(L(Q_X)))\to KK^{\mathbb T}_{*+1}(C(L(Q_X)),\mathbb C).    
\end{equation}
 
Douglas conjectured that $[\tau_X]$ is an equivariant fundamental class in $K$-homology, a statement we make precise below as the Fundamental Class (FC) conjecture. In terms of $KK$-theory, the conjecture can be reformulated as follows.
 
\begin{conj}[Fundamental class (FC) conjecture]
Let $X\subset \mathbb P^{d-1}$ be a projective variety. Assume that the Toeplitz extension $\tau_X$ defined in \eqref{eq:intro_toe} exists. Then, the map defined in $\eqref{eq:intro_duality}$ is an isomorphism. 
\end{conj}
 
This is an equivariant Poincar\'e duality statement in $KK$-theory, see Section \ref{sec:AD-conj_present}. According to the subsequent work of Douglas--Tang--Yu \cite{DTY16}, the FC conjecture holds in geometric situations where the shadow space $L(Q_X)$ is smooth, strongly pseudoconvex, and $[\tau_X]$ is represented by the Dirac operator. Also, due to the work in \cite{DTY16}, the FC conjecture suggests a canonical approach for extending the Grothendieck--Riemann--Roch theorem to singular varieties, see Remark \ref{rem: Proj_duality}.
 
However, as communicated to us by Xiang Tang (second named author in \cite{DTY16}), there is a singular counterexample to the FC conjecture, see Remark \ref{rem:counter_ex}. Nevertheless, there might be a way to fix this through the notion of Spanier--Whitehead duality in $KK$-theory, known as $KK$-duality. We discuss this in Open Problem 3. 
 
Despite the counterexample, the FC conjecture (or close alternatives as in Open Problem 3) still holds high mathematical value. In particular, the subproduct-system viewpoint also suggests that Toeplitz extensions should be regarded as a source of fundamental classes, even beyond the commutative setting, as happens, e.g. for Cuntz--Krieger algebras \cite{KP97} and Cuntz--Pimsner algebras \cite{AGN24, RRS19}.
 
Our second main result is that the FC conjecture is invariant under Veronese embeddings as well, even if the extensions $\tau_k$ are mutually different.

\begin{theoremC}[Theorem \ref{thm:Vpower_duality}]
Let $X\subset \mathbb P^{d-1}$ be a projective variety for which $Q_X$ is $\infty$-essentially normal, and $k\geq 1$ be arbitrary. Then, $$X\text{ satisfies FC conjecture}\Leftrightarrow \nu_k(X)\text{ satisfies FC conjecture}.$$
\end{theoremC}
 
Note that Theorem~\ref{thm:Vpower_duality}, combined with the counterexample in Remark \ref{rem:counter_ex}, produces infinitely many counterexamples to the FC conjecture.
 
\subsection*{Acknowledgements}
DMG would like to thank the organisers and participants of the workshop 
`Around the Arveson--Douglas conjecture' hosted in Potsdam in July 2026 
for several stimulating exchanges. In particular, he would like to thank 
Adam Dor-On for discussions related to Viselter's questions about the 
$K$-theory of Toeplitz algebras; Matthew Kennedy for discussions on 
hyperrigidity; Sven Raum and Xiang Tang for discussions related to duality. 
We are especially grateful to Xiang Tang for providing us with the 
counterexample in Remark \ref{rem:counter_ex}. FA and YG would like to 
thank Orr Shalit, Eli Shamovich and Yi Wang for insightful exchanges. 
The authors are also thankful to Magnus Goffeng and Bram Mesland for 
helpful comments and discussions.

\subsection*{Declaration on AI and AI-Assisted Technologies}

The authors acknowledge the use of Claude Sonnet 4.6 for checking spelling, punctuation, and for generating a draft of the present declaration. 
 
All mathematical content, proofs, results, and scientific conclusions are entirely the work of the authors, who bear full responsibility for the integrity and accuracy of the work. Our work was shaped by results of human-to-human interactions with the colleagues mentioned in the acknowledgement section. 

\section{Preliminaries}

\subsection{The Drury--Arveson space}
The Drury--Arveson space $H_d^2$ is a generalisation of the Hardy space to $d$ variables. For details, we refer to \cite{Har22}. Before we begin, we note that we will follow the usual convention for the inner products related to Drury--Arveson space, namely that they are linear in the first variable and conjugate-linear in the second. 

Formally, $H_d^2$ is defined as the Hilbert space completion of the ring of polynomials $\mathcal{R}_d=\Cmc$ with respect to the inner product 
\[\langle z^\alpha, z^\beta\rangle_{H_d^2}:=\delta_{\alpha,\beta} \frac{\alpha!}{|\alpha|!}.\] Here $\alpha,\beta\in \mathbb N^d$. Also, $\alpha!:=\alpha_1!\ldots \alpha_d!$ and $ |\alpha|=\alpha_1+\ldots \alpha_d$. Further, by $e_i\in \mathbb N^d$ we will denote the $d$-tuple $e_i(j):=\delta_{i,j}.$

Alternatively, the Drury--Arveson space admits a description as the space of holomorphic functions on the unit ball $f:\mathbb{B}^ d \to \mathbb{C}$ with  power series expansion $f(z)=\sum_{\alpha\in \mathbb N^d} c_{\alpha} z^{\alpha}$ satisfies \[\|f\|^2=\sum_{\alpha \in \mathbb N^d} |c_{\alpha}|^2 \frac{\alpha!}{|\alpha|!}<\infty.\]

An important feature of $H_d^2$ is that it is a reproducing-kernel Hilbert space (RKHS) on $\mathbb B_d$, with kernel $$K_{H_d^2}(z,w):=\frac{1}{1-\langle z,w\rangle_{\mathbb C^d}}.$$ Therefore, the linear span of $K_{H_d^2,w}:=K_{H_d^2}(\cdot,w)$ is dense in $H_d^2$. Moreover, \[\langle K_{H_d^2,z},K_{H_d^2,w}\rangle_{H_{d^2}}=K_{H_d^2}(w,z),\qquad \forall w, z \in \mathbb{B}_d\] and \[ \langle f, K_{H_d^2,w}\rangle_{H_d^2}=f(w), \quad \forall f\in H_d^2.\]


\subsection{Hilbert modules \'a la Arveson}
In his study of multivariable operator theory, Arveson \cite{WA07} introduced a notion of Hilbert space equipped with a polynomial ring action, under the name of \emph{Hilbert modules}. In what follows, we will refer to these objects as $\Rd$-modules, to avoid confusion with Hilbert $C^*$-modules. Recall that the symbol $\Rd$ stands for the ring $\Cmc$ of complex polynomials in $d$ variables.


\begin{definition}\label{defn:A-Hilbert}
	A Hilbert module over $\Rd$, also called a $\Rd$-module, is a Hilbert space $\HS$, equipped with a $\mathbb{C}$-algebra action given by a unital $\mathbb{C}$-algebra homomorphism \[ \sigma: \Rd \to \mathbb{B}(\HS). \]
	We will use the notation $(\HS, \sigma)$, or simply $\HS$ when the module action is clear. 
\end{definition}

Recall that a $d$-tuple of operators $(T_1, \dots, T_d)$ on a Hilbert space is called commuting if $[T_i,T_j]=0$ for all $1\leq i,j\leq d$.

Given a $\Rd$-module $(\HS, \sigma)$, it is straightforward to verify that the operators $(\sigma(z_i))_{i=1}^{d}$ form a commuting $d$-tuple acting on $\HS$. Conversely, given a commuting $d$-tuple $ (T_1, \dots, T_d)$ on $\HS$, the assignment $z_i \mapsto T_i$ extends uniquely to a unital $\mathbb{C}$-algebra homomorphism $\Rd \to \mathbb{B}(\HS)$.

Moreover, a $\mathcal{R}_d$-module $(\HS,\sigma)$ is called contractive if \[ \| \sum_{i = 1}^d \sigma(z_i)(\xi_i) \|^2 \leq \sum_{i = 1}^d \|\xi_i\|^2,\quad \forall \xi_i \in \HS. \] 


Amongst the contractive $\mathcal{R}_d$-modules, the Drury--Arveson space $H_d^2$ with the $\mathcal{R}_d$-action by multiplication operators, features prominently as a universal object (see e.g. \cite[Theorem 3.21]{Har22}). Specifically, consider the multiplication operators $M_{z_1},\ldots, M_{z_d}$ on $H_d^2$. It is straightforward to show that these are bounded, and a polynomial $p\in \mathcal{R}_d$ acts on $H_d^2$ as the operator $$M_p:=p(M_{z_1},\ldots, M_{z_d}).$$ 

Also, the reproducing-kernel feature asserts that the kernel functions are eigenfunctions for the adjoints of multiplication operators. Specifically, for any $p\in \mathcal{R}_d$ and $w\in \mathbb B_d$ we have that 
\begin{equation}\label{eq:adjoints}
M_p^*K_{H_d^2,w}=\overline{p(w)}K_{H_d^2,w}.
\end{equation}
To sum up, there is a canonical ring homomorphism $\sigma: \mathcal{R}_d\to \mathbb{B}(H_d^2).$


In the sequel, we also consider the following notion.

\begin{definition}
Let $(\HS,\sigma)$ be a $\mathcal{R}_d$-module, and let $\mathcal{P}$ be a closed subspace of $\HS$. We say that $\mathcal{P}$ is a \emph{submodule} of $\HS$ if \[ \sigma(q)(\mathcal{P}) \subset \mathcal{P}, \quad \forall q \in \mathcal{R}_d, \]
where the $\mathcal{R}_d$-action on $\mathcal{P}$ is inherited from the $\mathcal{R}_d$-action on $\HS$.

Given a submodule $\mathcal{P}$, we say $\mathcal{Q}$ is a \emph{quotient module} if, as a Hilbert space, $\mathcal{Q} = \HS / \mathcal{P}$, and the $\mathcal{R}_d$-action is given by \[ \sigma_\mathcal{Q}(q) = P_{\mathcal{Q}}  \sigma(q)|_\mathcal{Q},\quad \forall q \in \mathcal{R}_d,  \]
where $P_{\mathcal{Q}}$ is the orthogonal projection from $\HS$ onto $\mathcal{Q}$, and $\sigma(q)|_{\mathcal{Q}}$ is the restriction of the $\mathcal{R}_d$-action on $\HS$ to $\mathcal{Q}$.
\end{definition}

Of course, submodules and quotient modules of contractive modules, are again contractive. For an overview of the main differences and similarities between submodules and quotient modules, we refer the reader to \cite{GW08}. See, for instance, \cite[Proposition~1.04]{Wang18} for a characterisation of essential normality of quotient modules. 

\subsection{Geometric Arveson--Douglas conjecture}\label{sec:AD-conj_present}
 In his seminal paper \cite[Proposition~5.3]{WA97}, Arveson observed that the multiplication operators $M_{z_i}$ on the Drury--Arveson space $H^2_d$ give rise to a $p$-essentially normal $d$-tuple of operators, that is,
\[ [M_{z_i},M_{z_j}^*]\in \mathcal{S}_p(H^2_d),\]
for all $1\leq i,j\leq d$ and $p>d$. 

This led to the extension of the result to submodules and quotient modules of the Drury-Arveson space in subsequent work of Arveson in \cite{WA05} and of Douglas in \cite{RD06}. 

We will henceforth focus on homogeneous ideals  $I \triangleleft \mathcal{R}_d$ which are non-degenerate, meaning that $I\neq \mathcal R_d$ and $\sqrt{I}\neq (z_1,\ldots, z_d).$
Given a homogeneous ideal $I\triangleleft \mathcal{R}_d$ consider its closure $\overline{I}$ in $H_d^2$. There are two additional $\mathcal{R}_d$-modules at play, namely the submodule $\overline{I}$, and the quotient $Q_I:=H_d^2/\overline{I}$. The $\mathcal{R}_d$-action on the first is simply by restriction. 

For the second, the $\mathcal{R}_d$-action is given by compressing the multiplication operators $M_{z_j}$ to $$R_j:=P_{Q_I}M_{z_j}|_{Q_I}.$$ The operators $R_j$ will be used throughout the paper.

Given a non-degenerate ideal, we can assign a non-empty projective variety $V_{\mathbb{P}}(I):=\mathrm{Proj} (\mathcal{R}_{d}/I)$ and an affine variety $V_{\mathbb A}(I):=\mathrm{Spec}(\mathcal{R}_{d}/I)\neq \{0\}$. The latter is called the \textit{affine cone} of the former and $$V_{\mathbb P}(I)=(V_{\mathbb A}(I)\setminus \{0\})/\mathbb C^{\times}.$$ Note that affine cones in $\mathbb C^d$ are in one-to-one correspondence with radical homogeneous ideals in $\mathcal{R}_d$ through the Nullstellensatz. In the sequel, we will denote the affine cone of a projective variety $X\subset \mathbb P^{d-1}$ by $\widehat{X}\subset \mathbb{C}^{d}$.

\begin{conj}[Arveson--Douglas conjecture]\label{conj:ADmain}
Let $I\triangleleft \mathcal{R}_d$ be a homogeneous ideal. Then, the $\mathcal{R}_d$-module $Q_I$ is $p$-essentially normal, meaning that for every $1\leq i,j\leq d$ the commutators $$[R_i,R_j^*]\in \mathcal{S}_p(Q_I),$$ whenever $p>\dim_{\mathbb C} V_{\mathbb A}(I)$. 
\end{conj}

This conjecture is important because it opens the window to study the geometry of the affine variety $V_{\mathbb A}(I)$ through index theory. 

\begin{definition}
Let $I \triangleleft  \Rd$ be a homogeneous ideal and $Q_I$ the corresponding quotient module. We define the \emph{shadow space} of $V_{\mathbb P}(I)$ as
\[
L(Q_I) := V_{\mathbb{A}}(I) \cap \partial \mathbb{B}_d,
\]
\end{definition}
We will typically write $L(Q_I)$ to stress the relationship to the ideal 
$I$, and $L(Q_X)$ when we wish to stress the relationship to the underlying projective variety 
$X = V_{\mathbb P}(I)$.

Equipped with the relative topology from the unit sphere $\partial \mathbb{B}_d$, the shadow space $L(Q_{X})$ is a compact Hausdorff space. 
\begin{remark}
The set $L(Q_I)$ is a compact real semialgebraic set. 
When $X = V_{\mathbb P}(f) \subset \mathbb{P}^{d-1}$ is a hypersurface, it coincides 
with the \emph{link} of the isolated singularity of $\widehat{X}$ at the origin 
in the sense of Milnor \cite{Mil68}. In higher codimension, some of those are examples of links of complete intersections, see the monograph \cite{Loo84}.
\end{remark}

From \cite[Theorem~9.2.1]{BCR98}, $L(Q_I)$ has the homotopy type of a finite CW complex. Also, $V_{\mathbb A}(I)$ can be reconstructed from $L(Q_I)$ as $I$ is homogeneous. Specifically, we have the homeomorphism 
\begin{equation}\label{eq:salg_map}
 V_{\mathbb A}(I)\setminus \{0\}\cong L(Q_I)\times (0,\infty),\qquad z\mapsto \left(\frac{z}{\|z\|_{\mathbb C^d}}, \|z\|_{\mathbb{C}^d} \right).   
\end{equation} 
Consequently, since $2\dim_{\mathbb C} V_{\mathbb A}(I)$ equals the topological dimension $\dim V_{\mathbb A}(I)$ of $V_{\mathbb A}(I)\subset \mathbb C^d$ (see \cite[Proposition~2.8.2]{BCR98}), we get $$\dim L(Q_I)=2\dim_{\mathbb C} V_{\mathbb A}(I)-1.$$ Furthermore, following the definition \cite[Definition 2.1]{HHS23} of smooth/singular points for semi-algebraic sets and the map \eqref{eq:salg_map}, we obtain $$\mathrm{Sing}(L(Q_I))=\mathrm{Sing}(V_{\mathbb A}(I))\cap \partial \mathbb B_d.$$ Therefore, in general, $L(Q_I)$ can be quite singular. 

\begin{prop}[\cite{RD06}]\label{prop:Toeplitz_ext}
Assume the Arveson--Douglas conjecture holds for a homogeneous ideal $I\triangleleft \mathcal{R}_d$. Also, let $\mathcal{T}(Q_I):=C^*(R_j: 1\leq j\leq d)$. Then, $\mathcal{T}(Q_I)$ contains the ideal $\mathbb{K}(Q_I)$ of compact operators on $Q_I$ and the space $L(Q_I)$ fits into the short exact sequence \begin{equation}
\label{eq:toep}
\xymatrix{0 \ar[r] & \mathbb{K}(Q_I) \ar[r] & \mathcal{T}(Q_I) \ar[r]^{\sigma} & C(L(Q_I)) \ar[r] & 0}.\end{equation}
The symbol map $\sigma$ is given by $\sigma(R_j):=z_j|_{L(Q_I)},$ where $z_j$ is the $j$-coordinate function on $\mathbb C^d$. We will call this short exact sequence the \textit{Toeplitz extension of} $C(L(Q_I))$.
\end{prop}


\begin{rem}\label{rem:Equiv_Toeplitz}
Note that the Toeplitz extension of $C(L(Q_I))$ is $\mathbb{T}$-equivariant. Indeed, the Toeplitz algebra $\mathcal{T}(Q_I)$ admits a canonical $\mathbb{T}$-action defined on generators as $z\cdot R_j:=zR_j$, for $z\in \mathbb T$. This preserves $\mathbb{K}(Q_I)$. Moreover, $L(Q_I)$ admits a $\mathbb T$-action by scalar multiplication for which the symbol map $\sigma$ is $\mathbb T$-equivariant. Moreover, the $\mathbb T$-action on $L(Q_I)$ is free.
\end{rem}

The special version of the Arveson--Douglas conjecture deals with radical homogeneous ideals. This can be rephrased in terms of projective varieties through the Nullstellensatz. 

\begin{conj}
Let $X\subset \mathbb P^{d-1}$ be a projective variety. Then, the $\mathcal{R}_d$-module $Q_{I(X)}$ is $p$-essentially normal, whenever $p>\dim_{\mathbb C} \widehat{X}$.
\end{conj}

Another well-known feature of the radical setting is that $Q_I$ is also a reproducing-kernel Hilbert space. 

\begin{prop}[{e.g. \cite{Har22}}]
Let $I\triangleleft \mathcal{R}_d$ be radical and homogeneous. Then, $Q_I$ is a reproducing-kernel Hilbert space with kernel the restriction $K_{Q_I}$ of $K_{H_d^2}$ on $V_{\mathbb A}(I)\cap \mathbb B_d$. Therefore, $$Q_I\cong \overline{\mathrm{span}}\{K_{Q_I,w}:w\in V_{\mathbb A}(I)\cap \mathbb B_d\}.$$ 
\end{prop}

We now present the cases for which the conjecture has been verified. The conjecture is still open and no counterexamples have been found so far.

\begin{theorem}[Affirmative answers]\label{thm:aff_cases}
Let $I\triangleleft \Rd$ be a homogeneous ideal. The quotient module $\mathcal{Q}_I$ satisfies the Arveson--Douglas conjecture in the following cases:
\begin{enumerate}
    \item for monomial ideals \cite{WA05};
    \item for principal ideals, or $\dim_{\mathbb C}V_{\mathbb A}(I)\leq 1$, or $d\leq 3$ \cite{GW08};
    \item when $V_{\mathbb A}(I)$ is a finite union of linear subspaces by \cite{KS12};
    \item when $V_{\mathbb A}(I)$ is smooth away from the origin \cite{EE15}, \cite{DTY16}.
\end{enumerate}
\end{theorem}

 Note that in \cite{KS12}, it is proved that the conjecture is preserved by taking finite unions and the closedness assumption.

\section{Arveson--Douglas conjecture and Veronese embeddings}
In this section, we aim to show that the geometric Arveson--Douglas conjecture is invariant under Veronese embeddings. This approach bootstraps the class of known examples for which the conjecture holds. In particular, it leads to new highly singular examples. Also, it reduces the conjecture to the quadratic case.

To begin, let $k\geq 1$ and define the set $\Lambda_k:=\{\alpha \in \mathbb N^d:|\alpha|=k\}$ equipped with the lexicographic order. The $k$-Veronese embedding $\mathbb P^{d-1}\to \mathbb P^{N_k-1}$ is given by $$[z_1:\ldots :z_d]\mapsto [z^{\alpha}]_{\alpha\in \Lambda_k},$$ where $N_k={d+k-1\choose k}.$ For our purpose here, we will consider the normalised Veronese embeddings $\nu_k:\mathbb P^{d-1}\to \mathbb P^{N_k-1}$, given by 
\begin{equation}\label{eq:normVmap}
\nu_k([z_1:\ldots :z_d]):=[\lambda_k(\alpha)z^{\alpha}]_{\alpha\in \Lambda_k},\qquad \lambda_k(\alpha):=\sqrt{\frac{k!}{\alpha!}}.    
\end{equation}
Up to a linear change of coordinates, the maps $\nu_k$ are the same as the classical Veronese maps.

A classical fact relevant to our work is the following. 

\begin{lemma}\label{lem:Vdim}
Let $I\triangleleft \mathcal{R}_d$ be homogeneous and consider the graded algebra $A:=\mathcal{R}_d/I=\bigoplus_{n\geq 0} A_n$. Then, the $k$-th Veronese power $A^{(k)}:=\bigoplus_{n\geq 0} A_{kn}$ has the same Krull dimension as $A$. Consequently, $\dim_{\mathbb C} \mathrm{Spec}(A^{(k)})=\dim_{\mathbb C} \mathrm{Spec}(A).$ 
\end{lemma}

Another basic fact that generalises \cite[Proposition 5.3]{WA97} is the following. We present the proof as it is instructive.

\begin{lemma}\label{lem:VpowerAD}
Let $X\subset \mathbb P^{d-1}$ be a projective variety, and let $\mathcal{Q}_X$ denote its quotient module, with projection $P_{Q_X}$. Then, for every $1\leq i,j\leq d$ and $p>\dim_{\mathbb C} \widehat{X}$, the multiplication operators satisfy \[ P_{Q_X}[M_{z_i}^*,M_{z_j}]P_{Q_X}\in \mathcal{S}_p(Q_X).\] 
\end{lemma}

\begin{proof}
First, observe that $[M_{z_i}^*,M_{z_j}]$ is diagonal in the decomposition $H_d^2=\bigoplus_{n\geq 0}(H_d^2)_n$ into homogeneous parts. Let $P_n\in \mathbb B(H_d^2)$ be the projection onto $(H_d^2)_n$. Moreover, given $\alpha\in \mathbb N^d$, we have $$M_j(z^{\alpha})=z^{\alpha+e_j},\qquad M_i^*(z^{\alpha})=\frac{\alpha_i}{|\alpha|}z^{\alpha-e_i}.$$ In particular, simple computations give that, for every $n\geq 0$, it holds 
\begin{equation}\label{eq:comm_mult0}
[M_{z_i}^*,M_{z_j}]P_n=\frac{1}{n+1}(\delta_{i,j}I-M_{z_j}M_{z_i}^*)P_n.    
\end{equation}
As the multiplication operators in \eqref{eq:comm_mult0} are contractions, we obtain $\|[M_{z_i}^*,M_{z_j}]P_n\|_{\infty}=O(n^{-1}).$ Since $P_n$ commutes with $P_{Q_X}$ we then have
\begin{equation}\label{eq:comm_mult1}
\|P_nP_{Q_X}[M_{z_i}^*,M_{z_j}]P_{Q_X}P_n\|_{\infty}=O(n^{-1}).   
\end{equation}
Finally, consider the Hilbert series $H_X(t)=\sum_{n\geq 0}\dim_{\mathbb C}((Q_X)_n)t^n$ associated with $X$. Applying the Hilbert--Serre theorem \cite[I.7, Theorem 7.5.]{RH77} to $H_X$, gives that 
\begin{equation}\label{eq:comm_mult2}
\dim_{\mathbb C}((Q_X)_n)=O(n^{\dim_{\mathbb C} \widehat{X}-1}).
\end{equation}
Equations \eqref{eq:comm_mult1} and \eqref{eq:comm_mult2} imply the statement in the proof.
\end{proof}

Recall that $R_1,\ldots,R_d$ are the compressed multiplication operators implementing the $\mathcal{R}_d$-action on $Q_{X}.$ Also, note that $$\mathcal{R}_{N_k}=\mathbb C[w_{\alpha}:\alpha \in \Lambda_k],$$ where each $w_{\alpha}:\mathbb C^{N_k}\to \mathbb C$ is given by $w_{\alpha}((\xi_{\beta})_{\beta \in \Lambda_k})=\xi_{\alpha}.$


\begin{lemma}[{\cite[Lemma 9.16]{YG25}}] \label{lemma: veronese submodule structure}
Let $X\subset \mathbb P^{d-1}$ be a projective variety and $k\geq 1$. Then, the canonical $\mathcal{R}_{N_k}$-action on $Q_X$ leaves $Q_X^{(k)}:=\bigoplus_{l\geq 0} (Q_X)_{kl}$ invariant. In other words, $Q_X^{(k)}$ is a submodule of $Q_X$.

\end{lemma}

We refer to $Q_X^{(k)}$ as the k-th Veronese power of the module $Q$. For the general definition of Veronese powers of graded contractive $\Rd$-modules we refer to the dissertation~\cite[Section~9.3]{YG25}

We are now ready to state and prove our main result concerning the geometric Arveson--Douglas conjecture.

\begin{theorem}\label{thm:VpowerAD}
Let $X\subset \mathbb P^{d-1}$ be a projective variety, $k\geq 1$, and $p>\dim_{\mathbb C} \widehat{X}=\dim_{\mathbb C} \widehat{\nu_k(X)}.$ Then, the $\mathcal{R}_d$-module $Q_{X}$ is $p$-essentially normal if and only if the $\mathcal{R}_{N_k}$-module $Q_{\nu_k(X)}$ is $p$-essentially normal.
\end{theorem}

\begin{proof}
We first aim to build a Hilbert module isomorphism between $Q_{\nu_k(X)}$ and $ Q_{X}^{(k)}$. For this, note that the embedding $\nu_k:\mathbb P^{d-1}\to \mathbb P^{N_k-1}$ lifts to $\widehat{\nu_k}:\mathbb C^d\to \mathbb C^{N_k}$. The lift is a regular homogeneous map, but no longer injective for $k>1$. Moreover, using the multinomial theorem, for every $z,w \in \mathbb C^d$ we obtain that 
\begin{equation}\label{eq:VpowerAD_1}
\langle \widehat{\nu_k}(z),\widehat{\nu_k}(w)\rangle_{\mathbb C^{N_k}}=\sum_{\alpha \in \Lambda_k} \frac{k!}{\alpha!}z^{\alpha}\overline{w^{\alpha}}=\langle z,w \rangle_{\mathbb C^{d}}^k.
\end{equation}
Also, observe that $\widehat{\nu_k}(\widehat{X})=\widehat{\nu_k(X)}\subset \mathbb C^{N_k}$ and that $\widehat{\nu_k}(\widehat{X}\cap \mathbb B_d)=\widehat{\nu_k(X)}\cap \mathbb B_{N_k},$ hence 
\begin{equation}\label{eq:VpowerAD_2}
Q_{\nu_k(X)}=\overline{\mathrm{span}}\{K_{Q_{\nu_k(X)},\widehat{\nu_k}(w)}: w\in \widehat{X}\cap \mathbb B_d\}.
\end{equation}
Also, from \eqref{eq:VpowerAD_1} we have 
\begin{equation}\label{eq:VpowerAD_3}
K_{Q_{\nu_k(X)}}(\widehat{\nu_k}(z),\widehat{\nu_k}(w))=\frac{1}{1-\langle z,w\rangle_{\mathbb C^d}^k}.    
\end{equation}

On the other hand, the $k$-th Veronese power $Q_{X}^{(k)}$ of the subproduct system $Q_{X}$ is also reproduced by the kernel 
\begin{equation}\label{eq:VpowerAD_4}
K_{Q_{X}^{(k)}}(z,w)=\sum_{n\geq 0} \langle z,w \rangle_{\mathbb C^d}^{kn}=\frac{1}{1-\langle z,w\rangle_{\mathbb C^d}^k}.    
\end{equation}

Define now the linear map $U:Q_{\nu_k(X)}\to Q_{X}^{(k)}$ as
\begin{equation}\label{eq:VpowerAD_5}
U(K_{Q_{\nu_k(X)},\widehat{\nu_k}(w)}):=K_{Q_{X}^{(k)},w},\qquad w\in \widehat{X}\cap \mathbb B_d.
\end{equation}
This is a well-defined map and unitary since the kernels by \eqref{eq:VpowerAD_3} and \eqref{eq:VpowerAD_4} agree. We now check that $U$ can be made into a $\mathcal{R}_{N_k}$-module unitary. Let $R_1,\ldots,R_d$ be the compressed multiplication operators implementing the $\mathcal{R}_d$-action on $Q_{X}.$ Also, note that $$\mathcal{R}_{N_k}=\mathbb C[w_{\alpha}:\alpha \in \Lambda_k],$$ where each $w_{\alpha}:\mathbb C^{N_k}\to \mathbb C$ is given by $w_{\alpha}((\xi_{\beta})_{\beta \in \Lambda_k})=\xi_{\alpha},$ and in particular, $$w_{\alpha}(\widehat{\nu_k}(z))=\lambda_k(\alpha) z^{\alpha}, \qquad \lambda_k(\alpha)=\sqrt{\frac{k!}{\alpha !}}.$$ Then, the collection 
\begin{equation}\label{eq:Toep_gen}
 S_{\alpha}:=P_{Q_{\nu_k(X)}}M_{w_{\alpha}}|_{Q_{\nu_k(X)}},\qquad \alpha \in \Lambda_k   
\end{equation}
defines the canonical $\mathcal{R}_{N_k}$-action on $Q_{\nu_k(X)}$. Also, consider the collection 
\begin{equation}\label{eq:Toep_gen1}
T_{\alpha,(k)}:=\lambda_k(\alpha)R^{\alpha}|_{Q_{X}^{(k)}},\qquad \alpha\in \Lambda_k.    
\end{equation}
Here $R^{\alpha}:=R_1^{\alpha_1}\ldots R_d^{\alpha_d}$ and $T_{\alpha,(k)}:(Q_{X})_{kn}\to (Q_{X})_{k(n+1)}$ for every $n\geq 0$, so it acts on $Q_{X}^{(k)}.$ Moreover, from \eqref{eq:adjoints} we obtain that for $w\in \widehat{X}\cap \mathbb B_d$ it holds $$S_{\alpha}^*(K_{Q_{\nu_k(X)},\widehat{\nu_k}(w)})=\overline{w_{\alpha}(\widehat{\nu_k}(w)})K_{Q_{\nu_k(X)},\widehat{\nu_k}(w)}=\overline{\lambda_k(\alpha)w^{\alpha}}K_{Q_{\nu_k(X)},\widehat{\nu_k}(w)}.$$ Similarly, we have that $$T_{\alpha,(k)}^*(K_{Q_{X}^{(k)},w})=\overline{\lambda_k(\alpha)w^{\alpha}}K_{Q_{X}^{(k)},w}.$$ As a result, we have that 
\begin{equation}\label{eq:VpowerAD_6}
 US_{\alpha}U^*=T_{\alpha,(k)}.   
\end{equation}

($\Rightarrow$) Assuming that $Q_X$ is $p$-essentially normal, we claim $p$-essential normality for $Q_{\nu_k(X)}$ which, under \eqref{eq:VpowerAD_6}, is in terms of $\{T_{\alpha,(k)}\}_{\alpha \in \Lambda_k}$. For this, it suffices to show that 
\begin{equation}\label{eq:VpowerAD_7}
[R^{\alpha},(R^{\beta})^*]\in \mathcal{S}_p(Q_{X}^{(k)}),\qquad \alpha,\beta \in \Lambda_k.    
\end{equation}
Indeed, using the Leibniz rule for the commutators, each $[R^{\alpha},(R^{\beta})^*]\in \mathcal{S}_p(Q_{X})$ as it is a finite sum of products of bounded operators with some $[R_i,R_j^*]\in \mathcal{S}_p(Q_{X}),$ where $1\leq i,j\leq d$. Since $[R^{\alpha},(R^{\beta})^*]$ preserves $Q_X^{(k)}$, the assertion \eqref{eq:VpowerAD_7} follows.

($\Leftarrow$) Assume that $Q_{\nu_k(X)}$ is $p$-essentially normal, with $k\geq 2$. For $k=1$, nothing needs to be proved. Let $m\geq 1$ and $\alpha\in \Lambda_m$ be arbitrary, and define $M^{\alpha}:=M_{z_1}^{\alpha_1}\ldots M_{z_d}^{\alpha_d}$. Also, let $\widetilde{M}_{\alpha}:=\lambda_m(\alpha)M^{\alpha}$ and $\widetilde{M}_{\alpha,(m)}:=\widetilde{M}_{\alpha}|_{(H_d^2)^{(m)}}:(H_d^2)^{(m)}\to(H_d^2)^{(m)}.$ 

With respect to the decomposition $(H_d^2)^{(m)}=Q_X^{(m)}\oplus \overline{I}_X^{(m)}$, we have 
\begin{equation}\label{eq:VpowerAD_8}
 \widetilde{M}_{\alpha,(m)}=\begin{pmatrix}
    T_{\alpha, (m)}&0\\
    B_{\alpha,(m)}&C_{\alpha,(m)}
 \end{pmatrix}.   
\end{equation}
Here $T_{\alpha,(m)}$ is the operator in \eqref{eq:Toep_gen1}. Then, for $\alpha,\beta\in \Lambda_m$, using \eqref{eq:VpowerAD_8} and compressing by $P_{Q_X^{(m)}}$, we obtain 
\begin{equation}\label{eq:VpowerAD_9}
 P_{Q_X^{(m)}}[\widetilde{M}_{\alpha,(m)}^*,\widetilde{M}_{\beta,(m)}] P_{Q_X^{(m)}}=[T_{\alpha,(m)}^*,T_{\beta,(m)}] + B_{\alpha,(m)}^*B_{\beta,(m)}.  
\end{equation}
From Lemma \ref{lem:VpowerAD} we see, after using the Leibniz rule for the commutators, that the left hand side of \eqref{eq:VpowerAD_9} is in $\mathcal{S}_p(Q_X^{(m)})$. Therefore, $$[T_{\alpha,(m)}^*,T_{\beta,(m)}]\in \mathcal{S}_p(Q_X^{(m)}) \Leftrightarrow B_{\alpha,(m)}^*B_{\beta,(m)}\in \mathcal{S}_p(Q_X^{(m)}).$$ In particular, our assumption when $m=k$ and the unitary in \eqref{eq:VpowerAD_6}, give 
\begin{equation}\label{eq:VpowerAD_10}
B_{\alpha,(k)}^*B_{\beta,(k)}\in \mathcal{S}_p(Q_X^{(k)}),\qquad \alpha,\beta\in \Lambda_k.    
\end{equation}
We claim that \eqref{eq:VpowerAD_10} implies that 
\begin{equation}\label{eq:VpowerAD_11}
B_{\alpha,(1)}^*B_{\beta,(1)}\in \mathcal{S}_p(Q_X^{(1)}),\qquad \alpha,\beta\in \Lambda_1,    
\end{equation}
and hence $Q_X$ is $p$-essentially normal. Also, it clearly suffices to show \eqref{eq:VpowerAD_11} for $\alpha=\beta$. Therefore, we need a way to iterate \eqref{eq:VpowerAD_10} down to \eqref{eq:VpowerAD_11}.

To this end, we need an even finer decomposition than that in \eqref{eq:VpowerAD_8} of the multiplication operators $M^{\alpha}$, where $\alpha\in \Lambda_m$ and $m\geq 1$ are again arbitrary. Namely, for $n\geq 0$, consider the $n$-homogeneous decomposition $(H_d^2)_n=(Q_X)_n\oplus (\overline{I}_X)_n,$ with respect to which, the operator $\widetilde{M}_{\alpha,n}:=\widetilde{M}_{\alpha}|_{(H_d^2)_n}:(H_d^2)_n\to (H_d^2)_{n+m}$ is written as
\begin{equation}\label{eq:VpowerAD_12}
 \widetilde{M}_{\alpha,n}=\begin{pmatrix}
    T_{\alpha, n}&0\\
    B_{\alpha,n}&C_{\alpha,n}
 \end{pmatrix}.   
\end{equation}
Note that $\widetilde{M}_{\alpha,(m)}=\bigoplus_{l\geq 0} \widetilde{M}_{\alpha,ml},$ and similarly for the entries $T,B,C$. Also, define the positive operator 
\begin{equation}\label{eq:VpowerAD_13}
D_{m,n}:=\sum_{\alpha\in \Lambda_m}B_{\alpha,n}^*B_{\alpha,n}.
\end{equation}
Then, using that 
\begin{align*}
\widetilde{M}_{\beta}\widetilde{M}_{e_i}&=\sqrt{\frac{\beta_i+1}{m}}\widetilde{M}_{\beta+e_i},\qquad \beta\in \Lambda_{l},\,\, e_i\in \Lambda_1,\\
\sum_{i=1}^{d}\widetilde{M}_{e_i}^*\widetilde{M}_{e_i}|_{(H_d^2)_n}&=a_nI,\qquad a_n:=(n+d)(n+1)^{-1},
\end{align*}
(for the second equality see \cite[Proposition 5.3]{WA97}), we obtain that 
\begin{equation}\label{eq:VpowerAD_14}
 D_{m,n}=a_{n+1,m-1}D_{1,n} + \sum_{i=1}^{d}T_{e_i,n}^*D_{m-1,n+1}T_{e_i,n},\qquad a_{i,l}:=a_i\ldots a_{i+l-1}, \,\, a_{i,0}:=1.
\end{equation}
By induction with respect to $m$, \eqref{eq:VpowerAD_14} yields
\begin{equation}\label{eq:VpowerAD_15}
 D_{m,n}=\sum_{j=0}^{m-1}a_{n+j+1,m-j-1} \sum_{\beta\in \Lambda_{j}}T_{\beta,n}^*D_{1,n+j}T_{\beta,n}.
\end{equation}

As $\sum_{\beta\in \Lambda_{j}}T_{\beta,n}T_{\beta,n}^*=P_{(Q_X)_{n+j}}$, each $a_{i,l}\geq 1$ and each $T_{\beta,n}^*D_{1,n+j}T_{\beta,n}\geq 0$, for $0\leq j\leq m-1$, we get
\begin{align*}
    \|D_{1,n+j}\|_{\mathcal{S}_p}&=\|\sum_{\beta\in \Lambda_{j}}D_{1,n+j}^{1/2}T_{\beta,n}T_{\beta,n}^*D_{1,n+j}^{1/2}\|{\mathcal{S}_p}\\
    &\leq \sum_{\beta\in \Lambda_{j}}\|T_{\beta,n}^*D_{1,n+j}T_{\beta,n}\|_{\mathcal{S}_p}\\
    &\leq \# \Lambda_j\|D_{m,n}\|_{\mathcal{S}_p}.
\end{align*}
In particular, for $m=k$, $0\leq j\leq k-1$ and $n=kl$, with $l\geq 0$ be arbitrary, using that $D_{k,(k)}=\bigoplus_{l\geq 0} D_{k,kl}\in \mathcal{S}_p(Q_X^{(k)})$ from \eqref{eq:VpowerAD_10}, we have that 
\begin{equation*}\label{eq:VpowerAD_16}
\|D_{1,kl+j}\|_{\mathcal{S}_p}\leq \# \Lambda_j\|D_{k,kl}\|_{\mathcal{S}_p}<\infty.   
\end{equation*}
More importantly, as $D_{k,(k)}=\bigoplus_{l\geq 0} D_{k,kl}\in \mathcal{S}_p(Q_X^{(k)})$, we obtain that $\bigoplus_{n\geq 0} D_{1,n}\in \mathcal{S}_p(Q_X^{(1)})$, hence \eqref{eq:VpowerAD_11} holds, and the proof is complete.
\end{proof}

From Theorem \ref{thm:VpowerAD} and Proposition \ref{prop:Toeplitz_ext} we obtain a Toeplitz extension for each $L(Q_{\nu_k(X)})$, where $k\geq 1$. In order to understand those spaces, we present the following simple result.

\begin{prop}\label{prop:shadow_space}
Let $X\subset \mathbb P^{d-1}$ be a projective variety. Also, for $k\geq 1$ consider the group $\mu_k\subset \mathbb T$ of $k$-th roots of unity. Then, $$L(Q_{\nu_k(X)})\cong L(Q_{X})/\mu_k,\qquad \text{and}\qquad  L(Q_{\nu_k(X)})/\mathbb T\cong \nu_k(X)\cong X.$$ In particular, the canonical map $L(Q_{\nu_k(X)})\to \nu_k(X)$ is a $\mathbb T$-bundle.
\end{prop}



Applying Proposition \ref{prop:shadow_space} to $I=(0)\triangleleft \mathcal{R}_2$ we obtain the following $KK$-theoretic result. For details on $KK$-theory we refer to \cite{Bla98}.

\begin{cor}\label{cor:B}
For $k\geq 1$, the space $L(Q_{\nu_k(\mathbb P^1)})$ is homeomorphic to the lens space $L(k,1)$. Moreover, the Toeplitz algebra $\mathcal{T}(Q_{\nu_k(\mathbb P^1)})$ is not $KK$-equivalent to $\mathbb C$, when $k\geq 2$. 
\end{cor}

\begin{proof}
Applying Proposition \ref{prop:shadow_space} to $I=(0)\triangleleft \mathcal{R}_2$, we obtain the first part of the statement. The second follows from the Pimsner six-term exact sequence, which gives a surjection of $K_0(\mathcal{T}(Q_{\nu_k(\mathbb P^1)}))$ onto $K_0(L(Q_{\nu_k(\mathbb P^1)}))\cong \mathbb Z\oplus \mathbb Z_k$.     
\end{proof}

Below, we give a new example satisfying the Arveson--Douglas conjecture, as well as an example not covered purely by our mechanism.

\begin{example}[Cartan's umbrella]\label{ex:1} The cubic polynomial 
$q(z_1,z_2,z_3) = z_3 (z_1^2 + z_2^2) - z_1^3$ gives the so-called {\emph{Cartan's umbrella}}, a degree-3 affine surface with a line of singularities along the $x_3$-axis, from which the name \emph{umbrella}.

Set $X=V_{\mathbb P}(q).$ Then, the ideal $I(\nu_2(X)) \triangleleft  \mathbb{C}[y_0,\ldots,y_5]$ is generated by
nine quadrics (for simplicity, we display the generators for the non-normalised Veronese embedding):
\begin{align*}
  & g_1 = y_0 y_3 - y_1^2, &&
  g_2 = y_0 y_4 - y_1 y_2, &&
  g_3 = y_0 y_5 - y_2^2 \\
&  g_4 = y_1 y_4 - y_2 y_3, &&
  g_5 = y_1 y_5 - y_2 y_4, &&
  g_6 = y_3 y_5 - y_4^2 \\
  &f_1 = y_0^2 - y_2^2 - y_2 y_3 - y_4^2, &&
  f_2 = y_0 y_1 - y_1 y_2 - y_3 y_4, &&
  f_3 = y_0 y_2 - y_2^2 - y_4^2.
\end{align*}
The first six polynomials cut out the Veronese surface, while the last three arise from Cartan's equation. 

We see that this example is not covered by the existing theory as 
\begin{enumerate}
    \item the radical ideal $I(\nu_2(X))$ is not principal.
    \item The variety is a projective curve of degree $6$. Its affine cone satisfies $\dim_{\mathbb C} \widehat{\nu_2(X)}=2$. $d=6$;
    \item The variety $\widehat{\nu_2(X)}$ is not a union of finitely many linear subspaces, since $\widehat{\nu_2(X)}$ is irreducible and itself non-linear;
    \item The variety $\widehat{\nu_2(X)}$ is not smooth away from the origin.
\end{enumerate}
\end{example}

Now denote by $\mathbf{ProjVar_{\mathbb C}}$ the set of projective varieties over $\mathbb C$ and by $\mathbf{AE}$ the set of affirmative examples for which the conjecture has been verified, see Theorem~\ref{thm:aff_cases}. The Veronese operation offers a way to bootstrap $\mathbf{AE}$ to a strictly larger class $\mathbf{Ver(AE)}$ not restricted by dimension, number of generators, and smoothness. However, for now we have $$\mathbf{AE}\subsetneq \mathbf{Ver(AE)} \subsetneq \mathbf{ProjVar_{\mathbb C}}.$$

\begin{example}\label{ex:2}
For $d=2,k=3$, consider the twisted cubic $\nu_3(\mathbb P^1)\subset \mathbb P^3$. Recall that $L(Q_{\nu_3(\mathbb P^1)})\cong L(3,1)$. Also, consider the projective cone $\mathrm{Cone}(\nu_3(\mathbb P^1))\subset \mathbb P^4$; the projective closure of $\widehat{\nu_3(\mathbb P^1)}\subset \mathbb A^4$ in $\mathbb P^4$, built from homogenising its ideal with an additional variable. The following holds:
\begin{align*}
I(\nu_3(\mathbb P^1))&=(\sqrt{3}z_1z_3-z_2^2,3z_1z_4-z_2z_3,\sqrt{3}z_2z_4-z_3^2)\subset \mathcal{R}_4\\
J:=I(\mathrm{Cone}(\nu_3(\mathbb P^1)))&=(\sqrt{3}z_1z_3-z_2^2,3z_1z_4-z_2z_3,\sqrt{3}z_2z_4-z_3^2)\subset \mathcal{R}_5.
\end{align*}
We have that $X:=\mathrm{Cone}(\nu_3(\mathbb P^1))\notin \mathbf{AE}$. Indeed, $J$ is generated by quadratic polynomials and it is not principal. Also, $J$ is prime, hence $\widehat{X}$ is irreducible and so clearly not linear. Further, $\dim_{\mathbb C} \widehat{X}=3, d=5$ and $\widehat{X}$ is singular along the affine line in $\mathbb C^5$ of the new coordinate. 

Also, $X\notin \mathbf{Ver(AE)}$ since it cannot be a proper Veronese image (i.e. for $k>1$). Assume to the contrary that there is projective variety $Y$ such that $X=\nu_k(Y)$ for some $k\geq 2$. Then, from the numerator $h$ in Hilbert's series, we get $$\deg (X)=k^2 \deg(Y).$$ But since $\deg(X)=\deg(\nu_3(\mathbb P^1))=3$, this cannot happen.

Nevertheless, by direct computation, it is possible to check that the conjecture holds for $X$. We leave this as an exercise to the interested reader.
\end{example}

\subsection{Hilbert module perspective}
\label{thm:veronesePowers}We conclude this section with a brief discussion on how our result may be 
restated in terms of Hilbert modules, in lieu of projective varieties. 
Starting from a homogeneous ideal $I \triangleleft \mathbb{C}[z_1, \dots, z_d]$, 
one constructs the submodule $\overline{I} \subseteq H^2_d$ and then the 
quotient module $\mathcal{Q}:=H^2_d/\overline{I} = \bigoplus_{n = 0}^\infty Q_n$. 
From $\mathcal{Q}$ one can then proceed in two directions: either take its 
$k$-th Veronese power, the submodule $\mathcal{Q}^{(k)}:=\bigoplus_{n = 0}^\infty 
Q_{kn} \subseteq \mathcal{Q}$; or apply the Veronese embedding $\nu_k$ to 
the variety $V(I)$ and form the quotient module $Q_{\nu_k(V(I))}$ directly. 
It is important to distinguish these: $\mathcal{Q}^{(k)}$ is a submodule 
of $\mathcal{Q}$, whereas $Q_{\nu_k(V(I))}$ is a quotient module. The key 
point of Theorem~A is that these two are unitarily isomorphic as 
$\mathcal{R}_{N_k}$-modules, when $I$ is radical, and it is this isomorphism that transfers 
essential normality across Veronese embeddings. Using this algebraic perspective, and since essential normality, as a notion, is not restricted only to quotient modules  (see \cite{Wang18}), we can reformulate the result as follows.

\begin{theorem}
Let $k\geq 1$ and $p>\dim_{\mathbb{C}} V_{\mathbb A}(I)$. Then, $Q$ is a $p$-essentially normal $\mathcal{R}_d$-module, if and only if, $Q^{(k)}$ is a $p$-essentially normal $\mathcal{R}_{N_k}$-module.

\end{theorem}

\section{Fundamental classes and Veronese embeddings}\label{sec:KK-duality}
Let $X\subset \mathbb P^{d-1}$ be a projective variety such that the $\mathcal{R}_d$-module $Q_{X}$ is $\infty$-essentially normal; all commutators $[R_i,R_j^*]$ are simply in $\mathbb K (Q_X)$. Then, we have the Toeplitz extension of Proposition \ref{prop:Toeplitz_ext}. At this point, finer Schatten estimates are not required, as we focus on the Fundamental Class Conjecture of Douglas \cite{RD06}, which uses the language of $KK$-theory. This concept falls within the framework of $KK$-duality, which we discuss in Open Problem 3. For an introduction to Hilbert $C^*$-modules and Kasparov's bivariant K-theory, we refer to \cite{Bla98,JT91}.

To simplify the notation, we shall write $A:=C(L(Q_{X}))$ and denote by $\tau$ be its Toeplitz extension~\eqref{eq:toep}. Denote by $m$ the multiplication map on $A$. Then, we obtain the $K$-homology class $$\Delta:=m^*([\tau])\in KK^{\mathbb T}_1(A\otimes A,\mathbb C).$$ 

\begin{conj}[Fundamental Class conjecture]
For each $j=0,1$ the map, 
\begin{equation}\label{eq:fund_class}
-\otimes_{ A} \Delta:KK^{\mathbb T}_j(\mathbb C, A)\to KK^{\mathbb T}_{j+1}(A,\mathbb C),\qquad x\mapsto (x\otimes 1_A)\otimes_{A\otimes A}\Delta,
\end{equation}
is an isomorphism.
\end{conj}
According to \cite[Definition 5.2]{KS19}, such a class $[\tau]$ is indeed viewed as a fundamental class. Also, note that the image of $[1_A]\in KK_0^{\mathbb T}(\mathbb C,A)$ is $[\tau]$. However, the following holds.

\begin{remark}\label{rem:counter_ex}
The following counterexample to the fundamental class conjecture was communicated to us by Xiang Tang. Consider the ideal $I=(z_1,z_2)(z_3,z_4)\triangleleft \mathcal{R}_5$ and note that the Arveson--Douglas conjecture holds for $V_{\mathbb P}(I)$ since $V_{\mathbb A}(I)$ is a finite union of linear subspaces, see \cite{KS12}. Also, note that $L(Q_{I})$ is homeomorphic to the pushout $S^5\cup_{S^1}S^5$. We claim that the fundamental class conjecture is not satisfied, not even the non-equivariant version. Indeed, first observe that $S^5\cup_{S^1}S^5$ is homotopic to $S^5\vee S^5 \vee \Sigma S^1= S^5\vee S^5 \vee S^2$ since both inclusions $S^1\hookrightarrow S^5$ are null-homotopic. Also, we have the split short exact sequence:
\begin{equation}\label{eq:counter_ex}
 0\rightarrow C_0(S^5\setminus \{\mathrm{pt}\})\oplus C_0(S^5\setminus \{\mathrm{pt}\}) \oplus C_0(S^2\setminus \{\mathrm{pt}\})\rightarrow C(S^5\vee S^5 \vee S^2)\rightarrow \mathbb C\rightarrow 0,
\end{equation} 
and hence all the $K$-theory and $K$-homology groups of $L(Q_{I})$ are $\mathbb Z^2$. However, the Bott element $\beta \in KK_0(\mathbb C,C_0(S^2\setminus \{\mathrm{pt}\}))$, viewed as $(0,0,\beta,0)\in KK_0(\mathbb C, C(L(Q_{I})))$ in the $K$-theory decomposition from \eqref{eq:counter_ex} (we place the $K$-theory of the ideal to the left and of $\mathbb C$ to the right), has trivial pairing with $\Delta=m^*([\tau])$, as $[\tau]\in KK_1(C(L(Q_{I})),\mathbb C)$ is supported on the first and second coordinates and the Kasparov product is additive.
\end{remark}

Nevertheless, the FC conjecture still holds hope to be true in interesting cases, as in \cite{DTY16}. Here, we show that Veronese embeddings preserve the property of having a fundamental class from the Toeplitz extension. Similarly as before, for $k\geq 1$, let $A_k:=C(L(Q_{\nu_k(X)}))$ with multiplication map $m_k$, and the associated Toeplitz extension be $\tau_k$, which we know exists from Theorem \ref{thm:VpowerAD}. We then obtain another $K$-homology class $$\Delta_k:=m_k^*([\tau_k])\in KK^{\mathbb T}_1(A_k\otimes A_k,\mathbb C).$$
Note that $\Delta_1=\Delta$. To proceed, we construct a partial Green--Julg isomorphism. Although the statement is written in terms of classical equivariant $KK$-theory, for its proof, we pass through the twisted equivariant $KK$-theory of Chabert--Echterhoff \cite{CE01}. Also, we refrain from stating the lemma in a more general form, as it is not needed for the present work.

\begin{lemma}\label{lem:partialGJ}
Let $G$ be a compact group and $N$ be a finite normal subgroup of $G$. Also, let $B, D$ be separable unital $G$-$C^*$-algebras, with the restricted $N$-action being saturated on $B$ and trivial on $D$. 
Then, we have the isomorphism $$\Phi_{G,N}:KK_*^G(B,D)\to KK_*^{G/N}(B^N,D), \qquad [E,\pi, V, F]\mapsto [E^N,\pi|_{B^N},\overline{V}, P_NFP_N].$$ Here, 
\begin{itemize}
    \item $(E,\pi,F)$ is a Kasparov $B$-$D$-module, and its $G$-equivariant structure is dictated by the continuous $\mathbb C$-linear map $V:G\times E\to E$ which, for $g,h\in G,\,\, x,y\in E$ and $b\in B, d\in D$, satisfies $V_gV_h=V_{gh}$ and 
    $$V_g(xd)=V_g(x) (g\cdot d),\,\, g\cdot \langle x,y \rangle_D=\langle V_g(x),V_g(y)\rangle_D,\,\, V_g(\pi(b)x)=\pi(g\cdot b)V_g(x).$$ Also, the map from $G$ to the algebra of adjointable operators $\mathbb B_D(E)$, given by $g\mapsto (V_gFV_g^*-F)$ is norm-continuous with image in the compact operators $\mathbb K_D(E)$.
    \item $B^N$ is the fixed-point algebra of the $N$-action on $B$, $E^N:=\{\xi\in E: V_n(\xi) =\xi, \,\, n\in N\}$ and $$P_N=|N|^{-1}\sum_{n\in N} V_n\in \mathbb B_D(E)$$ is the orthogonal projection onto $E^N$. The $G/N$-equivariant structure on the Kasparov $B^N$-$D$-module $(E^N,\pi|_{B^N}, P_NFP_N)$ is given by $\overline{V}:G/N\times E^N\to E^N,\,\, (gN,\xi)\to V(g,\xi).$
\end{itemize}
\end{lemma}

\begin{proof}
As $N$ is finite, the full and reduced crossed products agree, so we simply write $B\rtimes N$. This crossed product admits a canonical $(G,N)$-action in the sense of Chabert--Echterhoff, see \cite[Example 2.10]{CE01}. Specifically, let $\{v_n, n\in N\}$ be the unitaries in $B\rtimes N$ implementing the $N$-action on $B$ by inner automorphisms. The $G$-action on $B\rtimes N$ is given by $$g\cdot (bv_n):=(g\cdot b)v_{gng^{-1}},\qquad g\in G, n\in N$$ and the twisting homomorphism from $N$ into the unitary group $U(B\rtimes N)$ is given by $n\mapsto v_n$, while the restricted $N$-action on $B\rtimes N$ is inner. 

To build the isomorphism $\Phi_{G,N}$ we first construct an isomorphism 
\begin{equation}\label{eq:partialGJ1}
 \mathrm{GJ_{G,N}^{\vee}}:KK_*^G(B,D)\to KK_*^{(G,N)}(B\rtimes N, D).   
\end{equation}
This is a partial version of the dual Green--Julg isomorphism; the proof is the usual integrated-form argument, see \cite[Propositions I.2.13–I.2.14]{Mas25}. Indeed, let $[E,\pi, V,F]\in KK^G(B,D)$ and assume that $\pi$ is non-degenerate. Since $N$ acts trivially on $D$, for $n\in N$, the operators $V_n$ are $D$-linear and so $V_n\in \mathbb B_D(E)$. Moreover, as $(E,\pi, V,F)$ is in particular $N$-equivariant, we have that $V_n\pi(b)V_n^*=\pi(n\cdot b).$ Therefore, $(\pi, V|_N)$ is a covariant representation of the $N$-action on $B$, meaning that we have an integrated non-degenerate representation 
\begin{equation}\label{eq:partialGJ2}
 \pi\rtimes V|_N:B\rtimes N\to \mathbb B_D(E).   
\end{equation}
Note that $\mathbb B_D(E)$ also admits a $(G,N)$-action; for $g\in G$ and $T\in \mathbb B_D(E)$ we define $g\cdot T:= V_gTV_g^*$. Also, the twisting homomorphism is $N\to U(\mathbb B_D(E)),\,\, n\to V_n.$ This turns the representation \eqref{eq:partialGJ2} into a $(G,N)$-equivariant one; meaning that the $G$-actions and the twists get intertwined. We can then define 
\begin{equation}\label{eq:partialGJ3}
\mathrm{GJ_{G,N}^{\vee}}([E,\pi, V,F]):=[E,\pi\rtimes V|_N, V,F].    
\end{equation}
The inverse of $\mathrm{GJ_{G,N}^{\vee}}$ is given by restricting the left $B\rtimes N$-action to just the $B$-action, since $N$ is discrete and $B\subset B\rtimes N$.

Moreover, we have an isomorphism 
\begin{equation}\label{eq:partialGJ4}
KK_*^{(G,N)}(B\rtimes N, D)\cong KK_*^{(G,N)}(B^N,D).    
\end{equation}
Indeed, let $p_N:=|N|^{-1}\sum_{n\in N} v_n\in B\rtimes N.$ Then, it is a standard fact (see \cite{Ros79}) that the map 
\begin{equation}\label{eq:partialGJ5}
 B^N\to p_N (B\rtimes N) p_N,\qquad b\to bp_N   
\end{equation}
is an isomorphism. Also, $B^N$ is invariant under the $G$-action on $B$ since $N$ is normal, which can be made into a $(G,N)$-action with trivial twisting homomorphism $n\mapsto 1_{B^N}$. Similarly, $p_N (B\rtimes N) p_N$ admits a $G$-action: for $g\in G$ and $bp_N\in p_N (B\rtimes N) p_N$ define $g\cdot (bp_N):=(g\cdot b)p_N$. This also becomes a $(G,N)$-action with trivial twist $n\mapsto p_N$. In particular, the isomorphism \eqref{eq:partialGJ5} is $(G,N)$-equivariant. Saturation of the $N$-action on $B$ means that $p_N$ is full, and hence the $E_N:=(B\rtimes N)p_N$ is a $(B\rtimes N)$-$B^N$ imprimitivity bimodule with a $(G,N)$-equivariant structure: the $G$-action on $E_N$ is the continuous $\mathbb C$-linear map $Z:G\times E_N\to E_N,\,\, (g,cp_N)\mapsto (g\cdot c)p_N$ and the (left) $B\rtimes N$-action on $E_N$ is $G$-equivariant and preserves the twists.

Therefore, we define the isomorphism $\Phi_{G,N}:KK_*^G(B,D)\to KK_*^{(G,N)}(B^N,D)$ by
\begin{equation}\label{eq:partialGJ6}
 \Phi_{G,N}([E,\pi,V,F]):=[E_N^*, \lambda,0]\otimes_{B\rtimes N}[E,\pi\rtimes V|_N, V,F].
\end{equation}
Note that $E_N^*=p_N(B\rtimes N)$. Also, $\lambda:B^N\to \mathbb B_{B\rtimes N}(E_N^*)$ is given by left multiplication. To compute the Kasparov product \eqref{eq:partialGJ6} we work as in \cite[Proposition 4.3]{Ger22}. We have that $$P_N:=(\pi\rtimes V|_N)(p_N)=|N|^{-1}\sum_{n\in N}V_n.$$ Also, observe that the map $W:E_N^*\otimes_{B\rtimes N} E\to P_NE,\,\, p_N x\otimes \xi\mapsto P_N ( \pi\rtimes V|_N)(x)\xi$ is unitary. Indeed, it is isometric with dense range since $\pi\rtimes V|_N$ is non-degenerate. Then, if $F'$ is an $F$-connection for $E_N^*$, we have that
\begin{equation}\label{eq:partialGJ7}
 \Phi_{G,N}([E,\pi,V,F])=[E_N^*\otimes_{B\rtimes N}E, \lambda\otimes 1, V_{\mathrm{diag}},F'],
\end{equation}
where $V_{\mathrm{diag}}$ is the diagonal $(G,N)$-action. Simple calculations show that we can choose $F'=W^{-1}P_NFP_NW$. Then, conjugating the $KK^{(G,N)}$-class with $W$ results to $[E^N,\pi|_{B^N},\overline{V}, P_NFP_N]$. Finally, $KK_*^{(G,N)}(B^N,D)\equiv KK_*^{G/N}(B^N,D)$ since the twists in both coordinates are trivial. The proof is now complete.
\end{proof}




The next result is now a consequence of Lemma \ref{lem:partialGJ}, functorial properties of the (dual) Green--Julg isomorphisms, and the unitary $U$ in \eqref{eq:VpowerAD_5} constructed in Theorem \ref{thm:VpowerAD}.

\begin{theorem}\label{thm:Vpower_duality}
Assume that $Q_X$ is $\infty$-essentially normal. Let $k\geq 1$. Then, the fundamental class conjecture holds for $A$ with fundamental class $\Delta$ if and only if it holds for $A_k$ with fundamental class $\Delta_k.$ 
\end{theorem}

\begin{proof}
Since $Q_X$ is $\infty$-essentially normal, the Toeplitz extensions $\tau_k$ of $A_k$ exist for every $k\geq 1$. Choose some $k\geq 1$ and set $G=\mathbb T, N=\mu_k$. The first claim is the existence of an isomorphism 
\begin{equation}\label{eq:Vpower_duality1}
\Psi_{G,N}: KK_*^G(\mathbb C, A)\to KK_*^{G/N}(\mathbb C,A^N).  
\end{equation}

We begin with the Green--Julg isomorphism (see \cite{Mas25})
\begin{equation}\label{eq:Vpower_duality2}
\mathrm{GJ_G} :KK_*^G(\mathbb C, A)\to KK_*(\mathbb C, A\rtimes G). 
\end{equation}
Also, the $(G,N)$-imprimitivity $(A\rtimes N)$-$A^N$-bimodule $E_N=(A\rtimes N)p_N$ from Lemma \ref{lem:partialGJ} yields an imprimitivity $(A\rtimes G)$-$(A^N\rtimes (G/N))$-bimodule $X_N$. Indeed, from \cite[Examples 2.8, 2.10]{CE01} we have 
\begin{align}
 A\rtimes G &\cong (A\rtimes N)\rtimes (G,N),\\
 A^N\rtimes (G/N)&\cong A^N\rtimes (G,N).
\end{align}
Then, from \cite[Proposition 4.1]{CE01} we obtain that $E_N$ yields such $X_N$. We define $$\Psi_{G,N}:=\mathrm{GJ}_{G/N}^{-1}\circ (-\otimes_{A\rtimes G} [X_N])\circ \mathrm{GJ}_G.$$ For a $G$-equivariant finitely generated projective module $M$ over $A$, we have $\Psi_{G,N}([M])=[M^N].$

Next, $G/N\cong G$ via $gN\mapsto g^k$. This makes $A_k$ a $G/N$-algebra that is $G/N$-equivariantly isomorphic to $A^N$, see Proposition \ref{prop:shadow_space}. Then, we have that $KK_*^{G/N}(A_k,\mathbb C)\cong KK_*^{G}(A_k,\mathbb C).$ Also, the isomorphism $\Phi_{G,N}:KK_1^G(A,\mathbb C)\to KK_1^{G}(A_k,\mathbb C)$ from Lemma \ref{lem:partialGJ} maps $[\tau]$ to $[\tau_k]$. 

Indeed, the extension class $[\tau]\in KK_1^G(A,\mathbb C)$ with $\tau:A\to \mathbb B(Q_X)/\mathbb K(Q_X)$ being the $G$-equivariant Busby invariant, has a $G$-equivariant Fredholm module representative $[H,F]$ (for brevity, we suppress the representation of $A$ on $H$ and the $G$-action), since any completely positive contractive lift of $\tau$ can be turned into a $G$-equivariant one through averaging, which then admits a $G$-equivariant Stinespring dilation through a $G$-equivariant isometry $J:Q_X\to H$, see \cite[Lemma 3.1]{RSW89}. Moreover, the projection (see Lemma \ref{lem:partialGJ}) $$P_N=k^{-1}\sum_{n\in N}V_n\in \mathbb B(H)$$ has $P_N(J(Q_X))=J(Q_X^{(k)})$, hence $\Phi_{G,N}$ maps $[\tau]$ to the class obtained by compressing $\tau$ by $P_N$. Then, conjugating the latter extension by the equivariant unitary $U$ in \eqref{eq:VpowerAD_5} gives exactly $\tau_k$.

Finally, the following diagram 
$$
\begin{tikzcd}
KK_j^G(\mathbb C,A)
\arrow[r,"-\otimes_A\Delta"]
\arrow[d,"\Psi_{G,N}"']
&
KK_{j+1}^G(A,\mathbb C)
\arrow[d,"\Phi_{G,N}"]
\\
KK_j^{G}(\mathbb C,A_k)
\arrow[r,"-\otimes_{A_k}\Delta_k"']
&
KK_{j+1}^{G}(A_k,\mathbb C),
\end{tikzcd}
$$
where $KK_*^{G/N}(\mathbb C,A^N)$ and $KK_*^{G/N}(A^N,\mathbb C)$ are identified with $KK_*^G(\mathbb C, A_k)$ and $KK_*^G(A_k,\mathbb C)$, respectively, is commutative. This follows from basic properties of the maps $\Phi_{G,N}$ and $\Psi_{G,N}$. Specifically, consider a $G$-equivariant finitely generated projective module $M$ over $A$. Note that $([M]\otimes 1_A)\otimes_{A\otimes A} [m]\in KK_0^G(A,A)$ equals the class of the Kasparov $A$-$A$-module $({}_AM_A,0)$, where ${}_AM_A$ is $M$ with left $A$-action being the right $A$-action due to $A$ being commutative. 

Then, working as in the proof of Lemma \ref{lem:partialGJ}, and using the fact that there is a unitary $W: ({}_AM_A\otimes_A H)^N\to ({}_AM_A)^N\otimes_{A^N} H^N$ (see \cite[Proposition 6.1]{BE15}), we compute
\begin{align*}
    \Phi_{G,N}([M]\otimes_A \Delta)&=\Phi_{G,N}\left ([M]\otimes 1_A)\otimes_{A\otimes A} [m]\otimes_A [\tau]\right)\\
    &=\Phi_{G,N}([{}_AM_A]\otimes_A [H,F])\\
    &= \Phi_{G,N}([{}_AM_A\otimes_A H,T]),\qquad \text{where } T \text{
     is an } F \text{-connection for }{}_AM_A\\
     &=[({}_AM_A\otimes_A H)^N,P_NTP_N]\\
     &= [({}_AM_A)^N\otimes_{A^N} H^N, T^N],\qquad \text{where } T^N:=WP_NTP_NW^*\\
     &= [({}_AM_A)^N]\otimes_{A^N}[H^N,F^N],\qquad \text{as } T^N \text{ is an } F^N \text{-connection for }({}_AM_A)^N\\
    &=\Psi_{G,N}([M])\otimes_{A^N} \Delta_k.
\end{align*}
This completes the proof.
\end{proof}

\begin{remark}\label{rem: Proj_duality}
All the $[\tau_k]$ yield the same canonical $K$-homology class of the projective variety $X$. Indeed, assume the map $F_{\Delta}:=-\otimes_{ A} \Delta$ in \eqref{eq:fund_class} is an isomorphism. Then, for each $k\geq 1$, the map $F_{\Delta_k}:=-\otimes_{ A_k} \Delta_k$ is an isomorphism as well. This gives a natural map $F_{X,k}:KK_1^{\mathbb T}(A_k,\mathbb C)\to KK_0(C(X),\mathbb C)$, taking $[\tau_k]$ to the $K^{\mathrm{top}}$-orientation class $\{X\}$ of Baum--Fulton--MacPherson \cite{BFM79}. Specifically, $F_{X,k}$ is defined as 
\small
$$KK_1^{\mathbb T}(A_k,\mathbb C)\xrightarrow{F_{\Delta_k}^{-1}} KK_0^{\mathbb T}(\mathbb C,A_k)\xrightarrow{\mathrm{GJ}_{\mathbb T}} KK_0(\mathbb C,A_k\rtimes \mathbb T)\xrightarrow{-\otimes [E_{\mathbb T}]} KK_0(\mathbb C,C(X))\xrightarrow{\cap \{X\}} KK_0(C(X),\mathbb C),$$ 
\normalsize
where $\mathrm{GJ}_{\mathbb T}$ is the Green--Julg isomorphism, $E_{\mathbb T}:=(A_k\rtimes \mathbb T)p_{\mathbb T}$ is the $(A_k\rtimes \mathbb T)$-$C(X)$ imprimitivity bimodule with $p_{\mathbb T}\in C(\mathbb T, A_k)\subset A_k\rtimes \mathbb T$ being the projection $p_{\mathbb T}(t)=1$. The image of $[\tau]$ under the composition of the first two maps is $[p_{\mathbb T}]$, which is mapped to $[1_{C(X)}]$ and consequently to $\{X\}$. Following \cite{DTY16}, the map $F_{X,k}$ is an isomorphism if $L(Q_X)$ is smooth, pseudoconvex. More details involving the Green--Julg isomorphism and imprimitivity bimodules can be found in the proofs of Lemma \ref{lem:partialGJ} and of Theorem \ref{thm:Vpower_duality}.
\end{remark}

\section{Open problems}
We conclude our work by listing a few problems related to the Veronese bootstrapping operation.

\begin{enumerate}
    \item An interesting fact is that blowup projections in algebraic geometry, as rational maps, are compositions of Veronese embeddings and linear projections \cite{JH92}. Do linear projections (thus blowup projections) preserve the geometric Arveson--Douglas conjecture as well? Potentially, this could reduce the Arveson--Douglas conjecture to the smooth away from zero case, which has already been confirmed in \cite{EE15}.
    \item Under the hypotheses for which $\infty$-essential normality is equivalent to hyperrigidity (see the work of Kennedy--Shalit \cite{KS15}), Theorem A and \cite[Theorem 4.12]{KS15} imply that, if the Veronese tuple \eqref{eq:Toep_gen}, $$S_{\alpha}:=P_{Q_{\nu_k(X)}}M_{w_{\alpha}}|_{Q_{\nu_k(X)}},\qquad \alpha \in \Lambda_k,$$
is hyperrigid in the Toeplitz algebra $\mathcal{T}(Q_{\nu_k(X)})=C^*(S_\alpha:\alpha\in\Lambda_k)$ for $k=1$, then it is hyperrigid for every $k\geq 1$. The Veronese operation is defined more generally for subproduct systems \cite{YG25}. If the canonical generating set is hyperrigid in its Toeplitz algebra, under what conditions is the same true for each of its Veronese powers?
   \item Regarding the counterexample in Remark \ref{rem:counter_ex}, the idea of Douglas that the Toeplitz extension yields a fundamental class might still be true, if we relax the need for self-duality. Namely, as $L(Q_X)$ has the homotopy type of a finite CW-complex, there is a Spanier--Whitehead dual space $Y$. The problem is to find a canonical such $Y$, ideally also admitting a $\mathbb T$-action. Further, it could be that asking for $Y$ to be a classical space is too restrictive, and instead we would have to search for a noncommutative one. In \cite{DL09}, the authors prove duality between stratified spaces and their noncommutative tangent spaces. Note that the counterexample in Remark \ref{rem:counter_ex} is stratified. In any case, assuming such a canonical equivariant dual $Y$ exists, it then seems natural to hope for $KK^{\mathbb T}$-duality between $L(Q_X)$ and $Y$. While this relaxes the requirement of self-duality, it leads to stronger homological consequences than the fundamental class conjecture alone. For the notion of $KK$-duality, see e.g. \cite{AGN24}.

   \item Given two projective varieties $X \subset \mathbb{P}^{m-1}$ and $Y \subset \mathbb{P}^{n-1}$, their Segre product $X \times Y$ embeds into 
$\mathbb{P}^{mn-1}$ via the classical Segre embedding. At the level of graded algebras, this corresponds to the Segre product 
operation on the associated coordinate rings \cite{PoPo, AGG22, GI23}. Does the geometric Arveson--Douglas conjecture \emph{behave well} under Segre 
products? More precisely, if $X$ and $Y$ satisfy the GAD conjecture, does 
$X \times Y$ satisfy it as well? The underlying dimensions suggests 
that the critical Schatten threshold for $X \times Y$ should be 
$p > \dim_{\mathbb{C}} \widehat{X} + \dim_{\mathbb{C}} \widehat{Y} - 1$, 
indicating that the Schatten degrees add up to a shift. A natural way to address this problem is considering the Segre product operation on
subproduct systems defined by the second author in \cite{YG25}.
\end{enumerate}

\end{document}